\documentclass{amsart}
\usepackage{amssymb}

\PassOptionsToPackage{pdfauthor={Vladimir V. Kisil},%
    pdftitle={},%
    pdfsubject={mathematics},%
    backref=page,%
  pdfkeywords={symmetries}
}{hyperref}
\IfFileExists{eulervm.sty}{\usepackage{eulervm}
}{}
\usepackage{amsthm}
\newtheorem{thm}{Theorem}
\newtheorem{prop}[thm]{Proposition}
\newtheorem{lem}[thm]{Lemma}
\newtheorem{cor}[thm]{Corollary}

\theoremstyle{definition}

\newtheorem{defn}[thm]{Definition}

\newtheorem{example}[thm]{Example}

\theoremstyle{remark}
\newtheorem{rem}[thm]{Remark}

\providecommand{\enorm}[2][\relax]{\left|\left|\left|#2\right|\right|\right|\ifx#1\relax\else_{#1}\fi}
\providecommand{\text}[1]{\mbox{#1}}
\providecommand{\submitto}[1]{\maketitle}
\let\ead=\email
\usepackage[british]{babel}
\usepackage[breaklinks=true,%
colorlinks=true,%
bookmarks=true,%
backref=page,%
pagebackref=true
]{hyperref}
\usepackage[backrefs,msc-links
,nobysame
]{amsrefs}

\input{mydef}

\begin{document}
\title{Operator Covariant Transform and Local Principle}

\author[Vladimir V. Kisil]%
{\href{http://www.maths.leeds.ac.uk/~kisilv/}{Vladimir V. Kisil}}
\address{School of Mathematics,
University of Leeds,
Leeds LS2\,9JT,
UK\\ 
On  leave from Odessa University}

\ead{\href{mailto:kisilv@maths.leeds.ac.uk}{kisilv@maths.leeds.ac.uk}}


\date{\today}

\begin{abstract}
  We describe connections between the localization technique introduced
  by I.B.~Simonenko and operator covariant transform produced by
  nilpotent Lie groups.
\end{abstract}
\submitto{\JPA}
\tableofcontents
\section{Introduction}
\label{sec:introduction}

In 1965 I.B.~Simonenko pioneered~\cite{Simonenko65a,Simonenko65b}
localization technique in the theory of operators. It still remains an
important tool in this area, see for
example~\cites{Vasilevski08a,BoetcherKarlovichSpitkovsky02a,DuduchavaSaginashviliShargorodsky97a,KarlovichSpitkovsky95a,RabinovichSamko11a,KarlovichSilbermann04a}.
Many questions addressed by this technique, e.g. boundary value
problems, are rooted in mathematical physics. We also discuss
connections with quantum mechanics in the closing section of this paper.

The localisation method was developed in various directions and there
is no possibility to mention all works based on numerous existing
variants and modifications of the localization technique.  Several
generalizations, e.g. within \(C^*\)-algebras
setup~\cite{Douglas72}*{Prop.~4.5}, capture the abstract skeleton of
the localization technique. However the idea of ``localization'' has
an explicit geometrical meaning, which often escapes those general
schemes.

We present here a different point of view on the original works of
Simonenko, which highlights the r\^ole of groups in the constructions.
Thus it is not a generalization but rather an attempt to link certain
geometrical meaning of locality with homogeneous structure of
nilpotent Lie groups.  This paper grown up from our earlier
works~\cite{Kisil94a,Kisil96e,Kisil92,Kisil93e,Kisil93b,Kisil94f,Kisil98a}
revised in the light of recent research~\cite{Kisil09d,Kisil10c}.

The paper outline is as follows: Section~\ref{sec:preliminaries}
collects preliminary information from other works, which will be used
here. In Section~\ref{sec:affine-group-local} we use homogeneous
structure of nilpotent Lie groups to define basic elements of
localization. Operators which are invariant under certain group action
are main building blocks for localization, we demonstrate this in
Section~\ref{sec:local-invar}. The final
Section~\ref{sec:closing-remarks} offers summary of our observations
which lead to new directions for further research.

\section{Preliminaries}
\label{sec:preliminaries}

\subsection{Classic Localization Technique}
\label{sec:class-local-techn}

We present here the fundamental definitions from the work of
I.B.~Simonenko~\cite{Simonenko65a,Simonenko65b} formulated for
operators on \(\FSpace{L}{p}(\Space{R}{n})\). Essential norm of an
operator is defined by
\begin{displaymath}
  \enorm{A}=\inf_K\norm{A-K},
\end{displaymath}
where the infimum is taken over all compact operators \(K\). For a
measurable set \(F\subset\Space{R}{n}\) we define the projection operator
\(P_F:\FSpace{L}{p}(\Space{R}{n})\rightarrow
\FSpace{L}{p}(\Space{R}{n})\) by:
\begin{equation}
  \label{eq:project-defn}
  [P_F f](x)=\left\{
    \begin{array}{ll}
      f(x),& \text{ if } x\in F;\\
      0,& \text{ otherwise}.
    \end{array}
  \right.
\end{equation}
The operators, most suitable for the localization method, are
defined as follows.
\begin{defn}\cite{Simonenko65a}*{\S~I.1}
  An operator \(A\) is of \emph{local type} if for any two closed disjoint sets
  \(F_1\) and  \(F_2\) the operator  \(P_{F_1} A P_{F_2}\) is compact.
\end{defn}
The cornerstone definition for the whole theory is
\begin{defn} \cite{Simonenko65a}*{\S~I.2}
  \label{de:local-equiv}
  Operators \(A\), \(B: \FSpace{L}{p}(\Space{R}{n})\rightarrow
  \FSpace{L}{p}(\Space{R}{n})\) are called \emph{equivalent} at a point
  \(x_0\) if for any \(\varepsilon>0\) there is a neighborhood \(u\)
  of \(x_0\) such that \(\enorm{AP_u-BP_u}<\varepsilon\) and
  \(\enorm{P_uA-P_uB}<\varepsilon\). This is denoted
  \(A\stackrel{x_0}{\sim} B\).
\end{defn}
As usual there are two stages in this method: analysis and synthesis.
Local equivalence decomposes operators into families of local
representatives. Now we define the opposite process of a reconstruction.

\begin{defn} \cite{Simonenko65a}*{\S~I.5}
  Let \(A_x\) be a family of operators \(\FSpace{L}{p}(X)\rightarrow
  \FSpace{L}{p}(X)\) depending from \(x\in X\). An operator
  \(A:\FSpace{L}{p}(X)\rightarrow \FSpace{L}{p}(X)\) is an
  \emph{envelope} of \(A_x\) if for every \(x\) we have
  \(A\stackrel{x}{\sim} A_x\).
\end{defn}
An envelope can be build \cite{Simonenko65a}*{\S~I.5} as the limit
\(A\) of a sequence \(A_n\) which is defined by the expression:
\begin{equation}
  \label{eq:envelope-approx}
  A_n=\sum_{j=1}^n P_{u_j} A_{x_j}  P_{u_j},
\end{equation}
where sets \(u_n\) make a decomposition of \(X\) and \(x_n\in u_n\). 

\subsection{Covariant Transform}
\label{sec:covariant-transform}
The following concept is a natural development of the coherent states
(wavelets) based on group representations.
\begin{defn}\cites{Kisil09d,Kisil10c}
  Let \(\uir{}{}\) be a representation of
  a group \(G\) in a space \(V\) and \(F\) be an operator from \(V\) to a space
  \(U\). We define a \emph{covariant transform}%
  \index{covariant!transform}%
  \index{transform!covariant}%
  \index{transform!wavelet|see{wavelet transform}}
  \(\oper{W}\) from \(V\) to the space \(\FSpace{L}{}(G,U)\) of
  \(U\)-valued functions on \(G\) by the formula:
  \begin{equation}
    \label{eq:coheret-transf-gen}
    \oper{W}: v\mapsto \hat{v}(g) = F(\uir{}{}(g^{-1}) v), \qquad
    v\in V,\ g\in G.
  \end{equation}
  Operator \(F\) will be called \emph{fiducial operator}%
  \index{fiducial operator}%
  \index{operator!fiducial} in this context.
\end{defn}
We borrow the name for operator \(F\) from fiducial vectors of Klauder
and Skagerstam~\cite{KlaSkag85}.  The wavelet transform, which is a
particular case of the covariant transform, corresponds to the
fiducial operator which is a linear functional. Thus its image
consists scalar-valued functions. It seems to be most favorable
situation, cf.~\cite{Kisil10c}*{Rem.~3}, and was believed to be the
only possible one for a long time.  A moral of the present work is
that the covariant transform can be useful even in the other extreme
limit: if the range of the fiducial operator is the entire space
\(V\).

By the way, we do not require that the fiducial
operator \(F\) shall be linear in general, however it will be always
linear in the present work.  Sometimes the positive homogeneity,
i.e. \(F(t v)=tF(v)\) for \(t>0\), alone can be already sufficient,
see~\cites{Kisil10c,Kisil11c}.

The following property is inherited by the coherent transform from the
wavelet one.
\begin{thm} \cites{Kisil09d,Kisil10c}
  \label{pr:inter1} 
  The covariant transform~\eqref{eq:coheret-transf-gen}
  intertwines%
  \index{intertwining operator}%
  \index{operator!intertwining} \(\uir{}{}\) and the left regular representation
  \(\Lambda\)    on \(\FSpace{L}{}(G,U)\):
  \begin{displaymath}
    \oper{W} \uir{}{}(g) = \Lambda(g) \oper{W}.
  \end{displaymath}
  Here \(\Lambda\) is defined as usual by:
  \begin{equation}\label{eq:left-reg-repr}
    \Lambda(g): f(h) \mapsto f(g^{-1}h).
  \end{equation}
\end{thm}
The next result follows immediately.
\begin{cor}\label{co:pi}
  The image space \(\oper{W}(V)\) is invariant under the
  left shifts on \(G\).
\end{cor}

\subsection{Inverse Covariant Transform}
\label{sec:inverse-covar-transf}

An object invariant under the left action
\(\Lambda\)~\eqref{eq:left-reg-repr} is called \emph{left invariant}.
For example, 
let \(L\) and \(L'\) be two left invariant spaces of functions on
\(G\).  We say that a pairing \(\scalar{\cdot}{\cdot}: L\times
L' \rightarrow \Space{C}{}\) is \emph{left invariant} if
\begin{equation}
  \scalar{\Lambda(g)f}{\Lambda(g) f'}= \scalar{f}{f'}, \quad \textrm{ for all }
  \quad f\in L,\  f'\in L'.
\end{equation}
\begin{rem}
  \begin{enumerate}
  \item We do not require the pairing to be linear in general.
  \item If the pairing is invariant on space \(L\times L'\) it is not
    necessarily invariant (or even defined) on the whole
    \(\FSpace{C}{}(G)\times \FSpace{C}{}(G)\).
  \item An invariant pairing on \(G\) can be obtained from an invariant
    functional \(l\) by the formula
    \(\scalar{f_1}{f_2}=l(f_1\bar{f}_2)\). Such a functional are
    often associated to the (quasi-) invariant measures.
  \end{enumerate}
\end{rem}
\begin{example}
  Let \(G\) be the \(ax+b\) group, cf. Ex.~\ref{ex:ax+b-appear} below.
  There are essentially two non-trivial invariant pairings for it. The
  first one is based on the left Haar measure \(\frac{da\,db}{a^2}\)
  and integration over the entire group:
  \begin{equation}
    \label{eq:haar-pairing}
    \scalar{f_1}{f_2}=
    \int\limits_{-\infty}^{\infty}\int\limits_0^{\infty}
    f_1(a,b)\,\bar{f}_2(a,b)\,\frac{da\,db}{a^2}.
  \end{equation}
  Another invariant pairing on \(G\), which is not generated by the Haar
  measure, is:
  \begin{equation}
    \label{eq:hardy-pairing}
    \scalar{f_1}{f_2}=
    \lim_{a\rightarrow 0}\int\limits_{-\infty}^{\infty}
    f_1(a,b)\,\bar{f}_2(a,b)\,db.
  \end{equation}
  This pairing participates in the definition of the inner product on
  the Hardy space, thus we call it \emph{Hardy-type
    pairing}~\cite{Kisil10c}. 
\end{example}

For a representation \(\uir{}{}\) of \(G\) in \(V\) and \(v_0\in V\)
we fix a function \(w(g)=\uir{}{}(g)v_0\). We assume that the pairing
can be extended in its second component to this \(V\)-valued
functions, say, in the weak sense.
\begin{defn}
  \label{de:admissible}
  Let \(\scalar{\cdot}{\cdot}\) be a left invariant pairing on
  \(L\times L'\) as above, let \(\uir{}{}\) be a representation of
  \(G\) in a space \(V\), we define the function
  \(w(g)=\uir{}{}(g)v_0\) for \(v_0\in V\). The \emph{inverse
    covariant transform} \(\oper{M}\) is a map \(L \rightarrow V\)
  defined by the pairing:
  \begin{equation}
    \label{eq:inv-cov-trans}
    \oper{M}: f \mapsto \scalar{f}{w}, \qquad \text{
      where } f\in L. 
  \end{equation}
\end{defn} 
There is an   easy consequence of this definition.
\begin{prop}
  The inverse wavelet transform intertwines the left regular
  representation and \(\uir{}{}(g)\).
\end{prop}

\section{Semidirect Products and Localization}
\label{sec:affine-group-local}

Let \(G\) be an \(m\)-dimensional exponential nilpotent Lie group of
the length \(k\). That means that
\begin{itemize}
\item we can identify \(G\) with its Lie algebra
  \(\algebra{g}\sim\Space{R}{m}\) through the exponential map;
\item there is a linear space decomposition 
  \begin{equation}
    \label{eq:nilpot-decomp}
    \algebra{g}=\oplus_{j=1}^{k} V_j, \qquad \text{ such that } \quad
    [V_i,V_j]\in V_{i+j},
  \end{equation}
  where \([V_i,V_j]\) denotes the
  space of all commutators \([x,y]=xy-yx\) with \(x\in V_i\), \(y\in
  V_j\)  and  \(V_l=\{0\}\) for all \(l>k\). 
\end{itemize}
\begin{example}
  \label{ex:nilpotent}
  Here are two most fundamental examples.
  \begin{enumerate}
  \item The group of Euclidean shifts in
    \(\Space{R}{n}\)---a nilpotent group of the length \(1\).
  \item The Heisenberg group
    \(\Space{H}{n}\)~\cites{Folland89,Howe80b}---a nilpotent group of
    dimensionality \(m=2n+1\) and the length \(2\). Its element is
    \((s,x,y)\), where \(x\), \(y\in \Space{R}{n}\) and
    \(s\in\Space{R}{}\). The group law on \(\Space{H}{n}\) is given as
    follows:
    \begin{equation}
      \label{eq:H-n-group-law}
      \textstyle
      (s,x,y)\cdot(s',x',y')=(s+s'+\frac{1}{2}(xy'-x'y),x+x',y+y').
    \end{equation} 
  \end{enumerate}
\end{example}
For a generic group \(G\) described above there is a one-parameter
group of automorphisms of \(\algebra{g}\) defined in terms of
decomposition~\eqref{eq:nilpot-decomp}:
\begin{displaymath}
  \tau_t (v_j) = t^j v_j,\qquad \text{ for } \quad v_j\in V_j,\ t\in\Space[+]{R}{}.
\end{displaymath}
The exponential map sends \(\tau_t\) to automorphisms of the group
\(G\) by the Baker--Campbell--Hausdorff formula. Thus we consider the
semidirect product \(\bar{G}=G \rtimes\Space[+]{R}{}\) of the group \(G\)
and positive reals with the group law:
\begin{displaymath}
  (t,g)\cdot (t',g')=(tt',g\cdot \tau_t(g')), \qquad \text{ where }
  t,t'\in\Space[+]{R}{}, \ g,g'\in G.
\end{displaymath}
The unit in \(\bar{G}\) is \((1,e)\) and \((t,g)^{-1}=(t^{-1}, \tau_{t^{-1}}(g^{-1}))\).
\begin{example}\label{ex:ax+b-appear}
  Returning to groups introduced in Example~\ref{ex:nilpotent}:
  \begin{enumerate}
  \item \label{item:ax+b-defn}
    If \(G\) is the group of shifts on the real line
    \(\Space{R}{}\) then the above semidirect \(\bar{G}\) product is the \(ax+b\)
    group (or \emph{affine} group). The group \(\bar{G}\) is
    isomorphic to \(\Space[+]{R}{} \times \Space{R}{}\) with the group
    law:
    \begin{displaymath}
      (a,b)\cdot (a',b')=(aa',ab'+b),\quad \text{ where }
      a,a'\in\Space[+]{R}{}, \ b,b'\in\Space{R}{}.
    \end{displaymath}
  \item For the Heisenberg group \(\Space{H}{n}\) the above
    automorphisms is \(\tau_t(s,x,y)=(t^2s,tx,ty)\)~\cite{Dynin75},
    thus the respective group law on \({\bar{\mathbb H}}^{n}\) is:
    \begin{equation}
      \label{eq:H-n-group-law}
      \hspace{-2cm}(t,s,x,y)\cdot(t', s',x',y')=(tt',s+t^2s'+\frac{t}{2}(xy'-x'y),x+tx',y+ty').
    \end{equation} 
  \end{enumerate}
\end{example}
There is a linear action of \(\bar{G}\) on 
functions over \(G\) cooked by the ``\(ax+b\)-recipe'': 
\begin{equation}
  \label{eq:Gbar-represent}
  [\uir{}{}(t,g)f](g')=t^{\frac{k}{p}}\, f(\tau_{t^{-1}}(g^{-1}\cdot g')),
\end{equation}
where \(k=\sum_j j\cdot\dim V_j\). This action is an isometry of
\(\FSpace{L}{p}(G)=\FSpace{L}{p}(G,d\mu)\), where \(d\mu\) is the Haar
measure on \(G\) (recall that it is unimodular as a nilpotent one).
Then we can define the respective representation \(\uir{}{d}\) of
\(\bar{G}\times \bar{G}\) on operators~\cites{Kisil98a,Kisil10c,Kisil11c}:
\begin{equation}
  \label{eq:uir-double-op}
  \uir{}{d}(t,g;t',g'): A \mapsto \uir{}{}(t^{-1},\tau_{t^{-1}}(g^{-1}))A\uir{}{}(t',g'),
\end{equation}
for a linear operator \(A:\FSpace{L}{p}(G)\rightarrow \FSpace{L}{p}(G)\).

Let \(F_e\subset G\) be a bounded closed subset, which contains a
neighbourhood of the unit \(e\in G\). We will denote by
\(F_{(t,g)}=(t,g)\cdot F_e\) for \((t,g)\in \bar{G}\), its image under
the left action of \(\bar{G}\) on \(G\).  Define the associated
projection \(P_e=P_{F_e}\) by~\eqref{eq:project-defn}. It is a
straightforward verification that
\begin{equation}
  \label{eq:shift-dilat-proj}
  \uir{}{d}(t,g;t,g) P_e= P_{F_{(t,g)}}, \qquad \text{ where }
  F_{(t,g)}=(t,g)\cdot F_e.
\end{equation}
We shall use a simpler notation \(P_{(t,g)}=P_{F_{(t,g)}}\) again. The
exact form of \(F_e\) is not crucial for the following construction,
but the following property simplifies technical issues:
\begin{defn}
  We say that \(F_e\) is \emph{\(r\)-self-covering} if for any two
  intersecting sets \(F_{(1,g_1)}\) and  \(F_{(1,g_2)}\) there is such
  \(g\in G\) that \(F_{(r,g)}\) covers the union of \(F_{(1,g_1)}\)
  and  \(F_{(1,g_2)}\) . 
\end{defn}
For example, the closed unit ball in \(\Space{R}{n}\) is
\(2\)-self-covering with no other \(F_e\) having a smaller value of
\(r\) for the self-covering property.

For a Banach space \(V\), we denote by \(B(V)\) the collection of
all bounded linear operators \(V\rightarrow V\).
\begin{defn}
  We select a fiducial operator \(F:B(\FSpace{L}{p}(G))\rightarrow
  B(\FSpace{L}{p}(G))\) by the identity
  \begin{equation}
    \label{eq:fiducial}
     F(A)=P_e A P_e, \qquad \text{ where } A\in B(\FSpace{L}{p}(G)).
   \end{equation} 
   Then \emph{Simonenko presymbol} \(\hat{S}_A(t,g;t',g')\) of an
   operator \(A\) is the covariant
   transform~\eqref{eq:coheret-transf-gen} generated by the
   representation \(\uir{}{d}\)~\eqref{eq:uir-double-op} and the
   fiducial operator \(F\)~\eqref{eq:fiducial}:
  \begin{eqnarray*}
    \hat{S}_A(t,g;t',g')&=&F(\uir{}{d}(t,g;t',g')A)\\
    &=&P_e\,  \uir{}{}(t^{-1},\tau_{t^{-1}}(g^{-1}))\,A\,\uir{}{}(t',g')\, P_e.
  \end{eqnarray*}
\end{defn}
Thus the Simonenko presymbol is \(B(\FSpace{L}{p}(G))\)-valued
function on \(\bar{G}\times\bar{G}\). We can consider a definition of
the alternative presymbol:
\begin{equation}
  \label{eq:alt-presymbol}
  \tilde{S}_A(t,g;t',g')   =P_{(t,g)}\,  A\, P_{(t',g')},
\end{equation}
which is closer to the original geometrical spirit of Simonenko's
works~\cite{Simonenko65a,Simonenko65b}. However there is an easy
explicit connection between them:
\begin{displaymath}
\hat{S}_A((t,g)^{-1};(t',g')^{-1})
=\uir{}{d}(t,g;t',g')\,\tilde{S}_A\,(t,g;t',g'),  
\end{displaymath}
which is a local transformation of the function value at every point.
Thus both symbols shall bring equivalent theories, although each of
them seems to be more suitable for particular purposes.

For operators of local type the whole presymbol  is excessive due
to the following result.
\begin{prop}
  Let \(F_e\) be \(r\)-self-similar and \(A\) be an operator of
  local type. Then for any reals \(t>t'>0\) and
  \(g\in G\) the operator \(\hat{S}_A(t_1,g_1;t_2,g_2)\) with \(t_i>t\),
  \(i=1,2\) can be expressed as a finite sum
  \begin{equation}
    \label{eq:decompos-less}
    \hat{S}_A(t_1,g_1;t_2,g_2)= \sum_{k=1}^n B_k\hat{S}_A(t',h_k;t',h_k) C_k,
  \end{equation}
  for some
  \(h_k\in F_{(t_1,g_1)} \cup F_{(t_2,g_2)}\)
  and constant operator coefficients  \(B_k\) and \(C_k\), which do
  not depend on \(A\).
\end{prop}
\begin{proof}
  We will proceed in terms of the equivalent presymbol
  \(\tilde{S}_A\)~\eqref{eq:alt-presymbol} since it better reflects
  geometrical aspects.  We also
  note that if we obtain the decomposition
  \begin{displaymath}
   \tilde {S}_A(t_1,g_1;t_2,g_2)= \sum_{k=1}^n B_k\tilde {S}_A(t_k,h_k;t_k,h_k) C_k,
  \end{displaymath}
  with all \(t_k\leq t'\) then we will be able to replace \(t_k\)
  by \(t'\) with the simultaneous change of coefficient \(B_k\)
  and \(C_k\) in order to get the required identity~\eqref{eq:decompos-less}.

  Now we put \(t''=t'/r\) and find a finite covering of the compact
  sets \(F_{(t_1,g_1)}\cup F_{(t_2,g_2)}\) by the interiors of sets
  \(F_{(t'',h_k)}\) with \(h_k\in F_{(t_1,g_1)} \cup F_{(t_2,g_2)}\).
  Using the inclusion-exclusion principle we can write:
  \begin{eqnarray*}
    P_{(t_i,g_i)}&=&\sum_k P_{(t'',h_k)} - \sum_{k,l} P_{(t'',h_k)}
    P_{(t'',h_l)}+
     \ldots\\
     &&{}-\sum_k P_{(t'',h_k)} {P}^\perp_{(t_i,g_i)}+ \sum_{k,l} P_{(t'',h_k)}
     P_{(t'',h_l)} {P}^\perp_{(t_i,g_i)}-
     \ldots,
  \end{eqnarray*}
  where all sums are finite and the number of sums is finite as well.
  Moreover each term in the summation contains at least one projection
  \(P_{(t'',h_k)}\). We use this decomposition for the presymbol
  \(P_{(t_1,g_1)} A P_{(t_2,g_2)}\) of an operator \(A\) of local
  type. Then we need to take care only on the terms \(P_{(t'',h_k)} A
  P_{(t'',h_l)}\) where \(F_{(t'',h_k)}\) and \(F_{(t'',h_n)}\)
  intersect. Due to the \(r\)-self-covering property each such term can be
  represented as \(B_m P_{(t',h_m)} A
  P_{(t',h_m)} C_m\) for some \(h_m\in F_{(t_1,g_1)}\cup
  F_{(t_2,g_2)}\) with \(B_m\) and \(C_m\) depending on the geometry of
  sets only.
\end{proof}
Thus for the operators of local type we give the following definition.
\begin{defn}
  For an operator \(A\) of local type we define \emph{Simonenko symbol}
  \(S_A(t,g)=\hat{S}_A(t,g;t,g)\), that is:
  \begin{eqnarray*}
    S_A(t,g)
    =P_e\,  \uir{}{}(t^{-1},\tau_{t^{-1}}(g^{-1}))\,A\,\uir{}{}(t,g)\, P_e.
  \end{eqnarray*}
\end{defn}
\begin{cor}
  For an operator \(A\) of local type the value of the presymbol
  \(\hat{S}_A(t',g';t'',g'')\) at a point \((t',g';t'',g'')\in
  \bar{G}\times\bar{G}\) is completely determined by the values of
  symbol \(S_A(t,g)\), \(g\in G\) for an arbitrary fixed \(t\) such that
  \(t\leq \min(t',t'')\).
\end{cor}
\begin{cor}
  \label{co:hardy-behaviour}
  Tho operators \(A\) and \(B\) of local type are equal if and only if
  for any \(\varepsilon>0\) there is a positive \(t<\varepsilon\) such
  that \(S_A(t,g)=S_B(t,g)\) for all \(g\in G\).
\end{cor}
In other words even the symbol \(S_A(t,g)\) contains an excessive
information: in a sense we shall look for values of
\(\lim_{t\rightarrow 0}S_A(t,g)\) only.  We conclude this section by the
restatement of the Definition~\ref{de:local-equiv}.
\begin{defn}
  Two operators \(A\) and \(B\) of local type are \emph{equivalent} at
  a point \(g\in G\), denoted by \(A\stackrel{g}{\sim}B\),  if
  \begin{displaymath}
    \lim_{t\rightarrow 0}\enorm{S_{A-B}(t,g)}=0.
  \end{displaymath}
\end{defn}

\section{Localization and Invariance}
\label{sec:local-invar}

The paper of Simonenko~\cite{Simonenko65b} already contains results
which can be easily adopted to covariant transform setup. This was
already used in our previous
work~\cite{Kisil94a,Kisil96e,Kisil92,Kisil93e,Kisil93b,Kisil94f} to
study singular integral operators on the Heisenberg group. In this
section we provide such restatements of results in term of the
representation from~\eqref{eq:Gbar-represent}. Proofs will be omitted
since they are easy modifications of the original ones~\cite{Simonenko65b}.
\begin{defn}
  An operator is called \emph{homogeneous} if it commutes with all
  transformations \(\uir{}{}(t,e)\),
  \(t\in\Space[+]{R}{}\)~\eqref{eq:Gbar-represent}. If an operator
  commutes with \(\uir{}{}(1,g)\),
  \(g \in G\)~\eqref{eq:Gbar-represent} then it is called \emph{shift-invariant}.
\end{defn}
There is an immediate consequence of Thm.~\ref{pr:inter1}.
\begin{cor}
  The symbols of a homogeneous (or shift-invariant) operator is a function
  on \(\bar{G}\), which is invariant under the action of the subgroup
  \(\Space[+]{R}{}\subset\bar{G}\) (or \(G\subset\bar{G}\) respectively).
\end{cor}
Thus homogeneous shift-invariant operators have constant symbols.
Tame behavior of operators from those classes is described by the
following statements, cf.~\cite{Simonenko65b}*{\S~II.2}.
\begin{lem} 
  For two homogeneous operators \(A\) and \(B\) the following are
  equivalent:
  \begin{enumerate}
  \item  \(A\stackrel{e}{\sim}B\), where \(e\in G\) is the unit;
  \item \(S_A(t,e)=S_B(t,e)\) for certain  \(t\in\Space[+]{R}{}\);
  \item \(A=B\).
  \end{enumerate}
\end{lem}

\begin{lem} \cite{Simonenko65b}*{\S~II.2} 
  For two homogeneous shift-invariant operators \(A\) and \(B\) the following are
  equivalent:
  \begin{enumerate}
  \item  \(A\stackrel{g}{\sim}B\) for certain \(g\in G\);
  \item \(S_A(t,g)=S_B(t,g)\) for certain  \((t,g)\in\bar{G}\);
  \item \(A=B\).
  \end{enumerate}
\end{lem}

A shift-invariant operator on \(G\) can be associated to a
convolution. A convolution, which is also a homogeneous operator, shall
have singular kernels. A study of such convolutions can be carried out
by means of (non-commutative) harmonic analysis on \(G\). For the
(commutative) Euclidean group this was illustrated
in~\cite{Simonenko65a,Simonenko65b}. A non-commutative example of the
Heisenberg group can be found
in~\cite{Kisil94a,Kisil96e,Kisil92,Kisil93e,Kisil93b,Kisil94f}. It is
also possible to study this operators through further versions of
wavelets (coherent) transform, e.g. the Berezin-type
symbols~\cite{Kisil98a}. In the common case boundedness of the Berezin
symbols corresponds to the boundedness of the operator, and if the
symbol vanishes at the infinity then the operator is compact. 

Once a good description of singular convolutions is obtained (through
covariant transform or several such transforms applied in a sequence)
we can consider the class of operators which can be reduced to them.
\begin{defn}
  \cite{Simonenko65b}*{\S~III.1}
  A linear operator \(A\) of local type is called a
  \emph{generalized singular integral} if \(A\) is equivalent at every
  point of \(G\) to a some homogeneous shift-invariant operator.
\end{defn}

The final step of the construction is synthesis of an operator from
the field of local representatives using the inverse covariant
transform from Subsection~\ref{sec:inverse-covar-transf}. To this end
we need to chose an invariant pairing on the group \(\bar{G}\),
keeping the \(ax+b\) group as an archetypal example. For operators of
local type the whole information is concentrated in the arbitrary
small neighborhood of the subgroup \(G\subset \bar{G}\), cf.
Cor.~\ref{co:hardy-behaviour}.  Thus we select the Hardy-type
functional~\eqref{eq:hardy-pairing} instead of the Haar
one~\eqref{eq:haar-pairing}.  Let \(d\mu\) be the Haar measure on the
group \(G\). Then the following integral
\begin{equation}
  \label{eq:hardy-integral}
  \scalar{f_1}{f_2}=\lim_{t\rightarrow 0} \int_G f_1(t,g) f_2(t,g)\,d\mu(g),
\end{equation}
defines an invariant pairing on the group \(\bar{G}\). 

We again make use of the fiducial operator \(F(A)=P_e A
P_e\)~\eqref{eq:fiducial}. In the language of wavelet theory we may
say that analyzing and reconstructing vectors are the same. The
respective transformation \(\uir{}{f}(t,g) F\) by an element of the
group \(\bar{G}\) is defined through the identity \([\uir{}{f}(t,g)
F](A)=P_{(t,g)} A P_{(t,g)}\) for an arbitrary \(A\). Consequently the
inverse covariant transform~\eqref{eq:inv-cov-trans} sends an operator
valued function \(A(t,g)\) to an operator through the invariant
pairing:
\begin{displaymath}
  \oper{M}: A(t,g)\mapsto A=\lim_{t\rightarrow 0} \int_G P_{(t,g)} A(t,g) P_{(t,g)}\,d\mu(g).
\end{displaymath}
The last integral may be realized through Riemann-type sums which are
lead to the approximation~\eqref{eq:envelope-approx} of an
envelope of \(A(t,g)\).

\section{Closing Remarks}
\label{sec:closing-remarks}

In this work we outlined an interpretation of the classical
Simonenko's localization method~\cites{Simonenko65a,Simonenko65b} in
the context of recently formulated covariant
transform~\cites{Kisil09d,Kisil10c}. The original localization
was used to study singular integral operators, which are convolutions
on the Euclidean group. Our interpretation allows to make a
straightforward modification of localization technique for
non-commutative nilpotent Lie groups. The crucial role is played by
the one-parameter group of automorphism realized as dilations.

Once local representatives are obtained they can be studied further by
other forms of wavelet (covariant) transform. The Berezin symbol seems
to be very suitable for this task. Such a chain
(Simonenko--Berezin--\ldots) of covariant transforms shall lead to
the full dissection of initial operator into a very detailed symbol,
which may be even scalar valued.  The opposite process, reconstruction
of an operator from its symbol or local representatives, can be done
by the inverse covariant transform, which uses the same group
structure.

The original coherent states in quantum mechanics are obtained from
the ground state of the harmonic oscillator by a unitary action of the
Weyl--Heisenberg group~\cite{AliAntGaz00}*{Ch.~1}. The next standard
move is a decomposition of an arbitrary state into a linear
superposition of coherent states, which form an overcomplete set.
Consequently, observables can be investigated through such
decompositions of states.

However, observables are primary notions of quantum theory, thus
direct techniques, which circumvent decomposition of states, look more
preferable.  Classical coherent states have the best possible (within
the Heisenberg uncertainty relations) localisation in the phase space.
Thus our localisation on nilpotent Lie groups, in particular the
Heisenberg group, has a particular significance for quantum theory.
Any observable corresponding to an operator of local type can be
represented as a compact operator and a continuous field of local
representatives. Compact operators have a discrete spectrum with a
complete set of eigenvectors each having at most a finite degeneracy.
Local representatives corresponds to observables which are highly
localised on the phase space. Thus operators of local type is a large
set of quantum observables admitting efficient calculations of their
spectrum.

It would be interesting to look for a similar construction in other
classes of Lie groups. For example, Toeplitz operators on the Bergman
space~\cite{Vasilevski08a} may be treated through the group
\(\SL\)~\cite{Kisil11c}, which is semisimple. Such groups do not admit
a group of dilation-type global automorphisms, thus some adjustments
to the scheme are required at this point.

Another interesting direction of development is operators of non-local
type. They may look very different from the view-point of geometrical
localization, however it terms of covariant transform the distinction
is not so huge. For operators of local type their Simonenko presymbol
over \(\bar{G}\times \bar{G}\) is excessive and we can consider only
the symbol in a small vicinity of the boundary \(G\) of the diagonal
in \(\bar{G}\times \bar{G}\). For operators of non-local type the
presymbol on the whole group \(\bar{G}\times \bar{G}\) shall be used.
This topic deserves a further consideration.\medskip

\textbf{Acknowledgements:} I am grateful to anonymous referees for
useful comments and suggestions, which helped to improve the paper.


\small

\bibliography{abbrevmr,akisil,analyse,algebra,arare,aclifford,aphysics}

\providecommand{\noopsort}[1]{} \providecommand{\printfirst}[2]{#1}
  \providecommand{\singleletter}[1]{#1} \providecommand{\switchargs}[2]{#2#1}
  \providecommand{\irm}{\textup{I}} \providecommand{\iirm}{\textup{II}}
  \providecommand{\vrm}{\textup{V}} \providecommand{\cprime}{'}
  \providecommand{\eprint}[2]{\texttt{#2}}
  \providecommand{\myeprint}[2]{\texttt{#2}}
  \providecommand{\arXiv}[1]{\myeprint{http://arXiv.org/abs/#1}{arXiv:#1}}
  \providecommand{\doi}[1]{\href{http://dx.doi.org/#1}{doi:
  #1}}\providecommand{\CPP}{\texttt{C++}}
  \providecommand{\NoWEB}{\texttt{noweb}}
  \providecommand{\MetaPost}{\texttt{Meta}\-\texttt{Post}}
  \providecommand{\GiNaC}{\textsf{GiNaC}}
  \providecommand{\pyGiNaC}{\textsf{pyGiNaC}}
  \providecommand{\Asymptote}{\texttt{Asymptote}}
\begin{bibdiv}
\begin{biblist}

\bib{AliAntGaz00}{book}{
      author={Ali, Syed~Twareque},
      author={Antoine, Jean-Pierre},
      author={Gazeau, Jean-Pierre},
       title={Coherent states, wavelets and their generalizations},
      series={Graduate Texts in Contemporary Physics},
   publisher={Springer-Verlag},
     address={New York},
        date={2000},
        ISBN={0-387-98908-0},
      review={\MR{2002m:81092}},
}

\bib{BoetcherKarlovichSpitkovsky02a}{book}{
      author={B{\"o}ttcher, Albrecht},
      author={Karlovich, Yuri~I.},
      author={Spitkovsky, Ilya~M.},
       title={Convolution operators and factorization of almost periodic matrix
  functions},
      series={Operator Theory: Advances and Applications},
   publisher={Birkh\"auser Verlag},
     address={Basel},
        date={2002},
      volume={131},
        ISBN={3-7643-6672-9},
      review={\MR{1898405 (2003c:47047)}},
}

\bib{Douglas72}{book}{
      author={Douglas, Ronald~G.},
       title={{B}anach algebra techniques in the theory of {Toeplitz}
  operators},
   publisher={American Mathematical Society},
     address={Providence, R.I.},
        date={{\noopsort{}}1972},
}

\bib{DuduchavaSaginashviliShargorodsky97a}{article}{
      author={Duduchava, R.},
      author={Saginashvili, A.},
      author={Shargorodsky, E.},
       title={On two-dimensional singular integral operators with conformal
  {C}arleman shift},
        date={1997},
        ISSN={0379-4024},
     journal={J. Operator Theory},
      volume={37},
      number={2},
       pages={263\ndash 279},
      review={\MR{1452277 (98d:47105)}},
}

\bib{Dynin75}{article}{
      author={Dynin, A.~S.},
       title={Pseudodifferential operators on the {H}eisenberg group},
        date={1975},
     journal={Dokl. Akad. Nauk SSSR},
      volume={225},
      number={6},
       pages={1245\ndash 1248},
      review={\MR{54 \#11410}},
}

\bib{Folland89}{book}{
      author={Folland, Gerald~B.},
       title={Harmonic analysis in phase space},
      series={Annals of Mathematics Studies},
   publisher={Princeton University Press},
     address={Princeton, NJ},
        date={1989},
      volume={122},
        ISBN={0-691-08527-7; 0-691-08528-5},
      review={\MR{92k:22017}},
}

\bib{Howe80b}{article}{
      author={Howe, Roger},
       title={Quantum mechanics and partial differential equations},
        date={1980},
        ISSN={0022-1236},
     journal={J. Funct. Anal.},
      volume={38},
      number={2},
       pages={188\ndash 254},
      review={\MR{83b:35166}},
}

\bib{KarlovichSilbermann04a}{article}{
      author={Karlovich, Yuri},
      author={Silbermann, Bernd},
       title={Fredholmness of singular integral operators with discrete
  subexponential groups of shifts on {L}ebesgue spaces},
        date={2004},
        ISSN={0025-584X},
     journal={Math. Nachr.},
      volume={272},
       pages={55\ndash 94},
         url={http://dx.doi.org/10.1002/mana.200310189},
      review={\MR{2079761 (2005d:47088)}},
}

\bib{KarlovichSpitkovsky95a}{article}{
      author={Karlovich, Yuri},
      author={Spitkovsky, Ilya},
       title={Factorization of almost periodic matrix functions},
        date={1995},
        ISSN={0022-247X},
     journal={J. Math. Anal. Appl.},
      volume={193},
      number={1},
       pages={209\ndash 232},
         url={http://dx.doi.org/10.1006/jmaa.1995.1230},
      review={\MR{1338508 (96m:47047)}},
}

\bib{Kisil92}{article}{
      author={Kisil, Vladimir~V.},
       title={Algebra of two-sided convolutions on the {H}eisenberg group},
        date={1992},
        ISSN={0869-5652},
     journal={Dokl. Akad. Nauk},
      volume={325},
      number={1},
       pages={20\ndash 23},
        note={Translated in \emph{Russ.Acad. of Sci Doklady, Math}, v. {\bf
  46}(1994), pp. 12--16. \MR{93j:22018}},
}

\bib{Kisil93b}{article}{
      author={Kisil, Vladimir~V.},
       title={On the algebra of pseudodifferential operators that is generated
  by convolutions on the {H}eisenberg group},
        date={1993},
        ISSN={0037-4474},
     journal={Sibirsk. Mat. Zh.},
      volume={34},
      number={6},
       pages={75\ndash 85, ii, viii},
        note={(Russian) \MR{95a:47053}},
}

\bib{Kisil94a}{article}{
      author={Kisil, Vladimir~V.},
       title={Local behavior of two-sided convolution operators with singular
  kernels on the {H}eisenberg group},
        date={1994},
        ISSN={0025-567X},
     journal={Mat. Zametki},
      volume={56},
      number={2},
       pages={41\ndash 55, 158},
        note={(Russian) \MR{96a:22015}},
}

\bib{Kisil93e}{article}{
      author={Kisil, Vladimir~V.},
       title={The spectrum of the algebra generated by two-sided convolutions
  on the {H}eisenberg group and by operators of multiplication by continuous
  functions},
        date={1994},
        ISSN={0869-5652},
     journal={Dokl. Akad. Nauk},
      volume={337},
      number={4},
       pages={439\ndash 441},
        note={Translated in \emph{Russ. Acad. of Sci Doklady, Math}, v. {\bf
  50}(1995), No 1, pp. 92--97. \MR{95j:22019}},
}

\bib{Kisil94f}{article}{
      author={Kisil, Vladimir~V.},
       title={Connection between two-sided and one-sided convolution type
  operators on a non-commutative group},
        date={1995},
        ISSN={0378-620X},
     journal={Integral Equations Operator Theory},
      volume={22},
      number={3},
       pages={317\ndash 332},
        note={\MR{96d:44004}},
}

\bib{Kisil96e}{article}{
      author={Kisil, Vladimir~V.},
       title={Local algebras of two-sided convolutions on the {H}eisenberg
  group},
        date={1996},
        ISSN={0025-567X},
     journal={Mat. Zametki},
      volume={59},
      number={3},
       pages={370\ndash 381, 479},
      review={\MR{MR1399963 (97h:22006)}},
}

\bib{Kisil98a}{article}{
      author={Kisil, Vladimir~V.},
       title={Wavelets in {B}anach spaces},
        date={1999},
        ISSN={0167-8019},
     journal={Acta Appl. Math.},
      volume={59},
      number={1},
       pages={79\ndash 109},
        note={\arXiv{math/9807141},
  \href{http://dx.doi.org/10.1023/A:1006394832290}{On-line}},
      review={\MR{MR1740458 (2001c:43013)}},
}

\bib{Kisil09d}{inproceedings}{
      author={Kisil, Vladimir~V.},
       title={Wavelets beyond admissibility},
        date={2010},
   booktitle={Proceedings of the 10$^{th}$ {ISAAC} {C}ongress, {L}ondon 2009},
   publisher={Imperial College Press},
       pages={6 pp.},
        note={\arXiv{0911.4701}},
}

\bib{Kisil10c}{article}{
      author={Kisil, Vladimir~V.},
       title={Covariant transform},
        date={2011},
     journal={Journal of Physics: Conference Series},
      volume={284},
      number={1},
       pages={012038},
         url={http://stacks.iop.org/1742-6596/284/i=1/a=012038},
        note={\arXiv{1011.3947}},
}

\bib{Kisil11c}{incollection}{
      author={Kisil, Vladimir~V.},
       title={{E}rlangen programme at large: an {O}verview},
        date={2012},
   booktitle={Advances in applied analysis},
      editor={Rogosin, S.V.},
      editor={Koroleva, A.A.},
       pages={1\ndash 65},
        note={\arXiv{1106.1686}},
}

\bib{KlaSkag85}{book}{
      editor={Klauder, John~R.},
      editor={Skagerstam, Bo-Sture},
       title={Coherent states. {A}pplications in physics and mathematical
  physics.},
   publisher={World Scientific Publishing Co.},
     address={Singapur},
        date={{\noopsort{}}1985},
}

\bib{RabinovichSamko11a}{article}{
      author={Rabinovich, Vladimir},
      author={Samko, Stefan},
       title={Pseudodifferential operators approach to singular integral
  operators in weighted variable exponent {L}ebesgue spaces on {C}arleson
  curves},
        date={2011},
        ISSN={0378-620X},
     journal={Integral Equations Operator Theory},
      volume={69},
      number={3},
       pages={405\ndash 444},
         url={http://dx.doi.org/10.1007/s00020-010-1848-x},
      review={\MR{2774611}},
}

\bib{Simonenko65a}{article}{
      author={Simonenko, I.~B.},
       title={A new general method of investigating linear operator equations
  of singular integral equation type. {I}},
        date={1965},
        ISSN={0373-2436},
     journal={Izv. Akad. Nauk SSSR Ser. Mat.},
      volume={29},
       pages={567\ndash 586},
      review={\MR{0179630 (31 \#3876)}},
}

\bib{Simonenko65b}{article}{
      author={Simonenko, I.~B.},
       title={A new general method of investigating linear operator equations
  of singular integral equation type. {II}},
        date={1965},
        ISSN={0373-2436},
     journal={Izv. Akad. Nauk SSSR Ser. Mat.},
      volume={29},
       pages={757\ndash 782},
      review={\MR{0188738 (32 \#6174)}},
}

\bib{Vasilevski08a}{book}{
      author={Vasilevski, Nikolai~L.},
       title={Commutative algebras of {T}oeplitz operators on the {B}ergman
  space},
      series={Operator Theory: Advances and Applications},
   publisher={Birkh\"auser Verlag},
     address={Basel},
        date={2008},
      volume={185},
        ISBN={978-3-7643-8725-9},
      review={\MR{2441227 (2009m:47071)}},
}

\end{biblist}
\end{bibdiv}
\end{document}